\documentclass[twoside,pdf,11pt]{amsart}
\usepackage{amsfonts, amsbsy, amsmath, amsthm, amssymb, latexsym}
\usepackage{mathrsfs}
\usepackage{enumerate,centernot}
\usepackage[top=30mm,right=30mm,bottom=30mm,left=30mm]{geometry}
\usepackage{latexsym}
\usepackage{longtable}
\usepackage{epsfig}
\usepackage{hhline}
\usepackage{amscd}
\usepackage{newlfont}
\usepackage{enumerate}
\usepackage{tabularx}
\usepackage{multirow}
\usepackage{amsfonts, amsbsy, amsmath, amsthm, amssymb, latexsym, verbatim, enumerate}
\usepackage{mathrsfs}
\usepackage[top=30mm,right=30mm,bottom=30mm,left=30mm]{geometry}

\usepackage{bm}

\usepackage{bm}
\usepackage{fancyhdr}

\numberwithin{equation}{section}
\newcommand{\core}{\mathop{\mathrm{core}}}

\newcommand{\Aut}{\mathop{\mathrm{Aut}}}

 \newcommand{\al}{\alpha}

\newtheorem{Theorem}{Theorem}[section]

\newtheorem{Proposition}[Theorem]{Proposition}
\newtheorem{Lemma}[Theorem]{Lemma}
\newtheorem{Corollary}[Theorem]{Corollary}

\theoremstyle{definition}

\begin{document}

\title{On the generalized Fitting height and insoluble length of finite groups}

\author{Robert M. Guralnick}
\address{Department of Mathematics, University of Southern California,
Los Angeles, CA 90089-2532, USA}
\email{guralnic@usc.edu}

\author{Gareth Tracey}
\address{Alfr\'ed R\'enyi Institute of Mathematics, Hungarian Academy of Sciences, Re\'altanoda utca
13-15, H-1053, Budapest, Hungary}
\email{g.m.tracey6@gmail.com}

\thanks{The first author was partially supported by the NSF
grant DMS-1901595, while the second author has received funding from the European Research Council (ERC) under the European Unions Horizon 2020 research and
innovation programme (grant agreement No. 741420)}
\subjclass[2010]{20D30,20D35,20F45}
\begin{abstract} We prove two conjectures of E. Khukhro and P. Shumyatsky concerning the Fitting height and insoluble length of finite groups. As a by-product of our methods, we also prove a generalization of a result of Flavell, which itself generalizes Wielandt's Zipper Lemma and provides a characterization of subgroups contained in a unique maximal subgroup.  We also derive a number of consequences of our theorems, including some applications to the set of odd order elements of a finite group inverted by an involutory automorphism.  
 \end{abstract}
\maketitle

\section{Introduction}
A classical result of R. Baer \cite{Baer} states that an element $x$ of a finite group $G$ is contained in the Fitting subgroup $F(G)$ of $G$ if and only if $x$ is a left Engel element of $G$. That is, $x\in F(G)$ if and only if there exists a positive integer $k$ such that $[g,_k x]:=[g,x,\hdots,x]$ (with $x$ appearing $k$ times) is trivial for all $g\in G$. (In this paper, we use left normed commutators, so that $[x_1,x_2,x_3\hdots,x_k]:=[[\hdots[[x_1,x_2],x_3],\hdots],x_k]$). The result was generalized by E. Khukhro and P. Shumyatsky in \cite{KS} via an analysis of the sets
\begin{align*}E_{G,k}(x) &:=\{[g,_k x]\text{ : }g\in G\}. \end{align*}
In this notation, Baer's Theorem states that $x\in F(G)$ if and only if $E_{G,k}(x)=\{1\}$ for some positive integer $k$. The generalization of Khukhro and Shumyatsky takes three directions. First, if $G$ is soluble then a complete generalization is obtained: \cite[Theorem 1.1]{KS} proves that if the Fitting height of the subgroup $\langle E_{G,k}(x)\rangle$ (for any $k$) is $h$, then $x$ is contained in $F_{h+1}(G)$ --- the $(h+1)$-st Fitting subgroup of $G$. 

Secondly, they also discuss analogous results for insoluble groups: for a finite group $G$, we will write $F^{\ast}(G)$ for the generalized Fitting subgroup of $G$. That is, $F^{\ast}(G):=F(G)E(G)$, where $F(G)$ is the Fitting subgroup of $G$, and 
$$E(G):=\langle N \text{ : } N\text{ a quasisimple subnormal subgroup of } G\rangle$$ is the \emph{layer} of $G$. We then write $F^{\ast}_i(G)$ for the $i$-th generalized Fitting subgroup of $G$. That is, $F^{\ast}_1(G):=F^{\ast}(G)$ is the generalized Fitting subgroup of $G$, and $F^{\ast}_i(G)$ is the inverse image of $F^{\ast}(G/F^{\ast}_{i-1}(G))$ in $G$ for $i\geq 2$. Thus, in particular, $F_i(G)=F^{\ast}_i(G)$ when $G$ is soluble. It is proven in \cite[Theorem 1.2]{KS} that if $\langle E_{G,k}(x)\rangle$ has generalized Fitting height $h$, then $x$ is contained in $F^{\ast}_{f(x,h)}(G)$ for  a certain function $f$ defined in terms of $h$ and the number of prime divisors of the order of $x$ (counting multiplicities). The authors conjecture in \cite[Conjecture 7.1]{KS} that in this case, $x$ is in fact contained in $F^{\ast}_{h+1}(G)$. The first main result of this paper is a proof of this conjecture.
\begin{Theorem}\label{KSConjecture1} Let $G$ be a finite group, let $x$ be an element of $G$, and fix $h\geq 0$. Then $F^{\ast}_h(G)x$ is contained in $F(G/F^{\ast}_{h}(G))$ if and only if
$\langle E_{G,k}(x)\rangle$ has generalized Fitting height at most $h$ for some positive integer $k$. In particular, if $\langle E_{G,k}(x)\rangle$ has generalized Fitting height $h$ for some positive integer $k$, then $x$ is contained in $F^{\ast}_{h+1}(G)$.\end{Theorem}
In fact, since it is an ``if and only if" statement, Theorem \ref{KSConjecture1} is stronger than \cite[Conjecture 7.1]{KS}. In particular, Baer's classical result can be recovered. 

Thirdly, another length parameter for finite groups is discussed. For a finite group $G$, write $\lambda(G)$ for the \emph{insoluble length} of $G$. That is, $\lambda(G)$ is the minimum number of insoluble factors in a normal series for $G$ each of whose factors is either soluble or a direct product of non-abelian simple groups. In particular, a group is soluble if and only if $\lambda(G)=0$. The group $R_0(G)$ is defined to be the soluble radical of $G$, while $R_i(G)$ is defined to be the largest normal subgroup of $G$ with insoluble length $i$, for $i\geq 1$. The series $1\le R_0(G)\le \hdots\le R_i(G)\le \hdots\le G$ is called the \emph{upper insoluble series} for $G$, and \cite[Theorem 1.3]{KS} shows that if $\langle E_{G,k}(x)\rangle$ has insoluble length $h$, then $x$ is contained in $R_{r(x,h)}(G)$ for a certain function $r$ defined in terms of $h$ and the number of prime divisors of $x$ (again counting multiplicities). Khukhro and Shumyatsky conjecture in \cite[Conjecture 7.2]{KS} that we should have $x\in R_h(G)$ in this case, and our next result proves their conjecture.     
\begin{Theorem}\label{KSConjecture2} Let $G$ be a finite group, let $x$ be an element of $G$, and fix $h\geq 0$. Then $x$ is contained in $R_{h}(G)$ if and only if $\langle E_{G,k}(\alpha)\rangle$ has insoluble length at most $h$ for some positive integer $k$.\end{Theorem}
Again, since Theorem \ref{KSConjecture2} is an ``if and only if" statement, it is stronger than \cite[Conjecture 7.2]{KS}.

Bounds on the insoluble length and the Fitting height of a finite group have proved to be powerful tools in both finite and profinite group theory. In particular, such bounds were crucial in the reduction of the Restricted Burnside Problem to soluble and nilpotent groups due to P. Hall and G. Higman \cite{HH}. J. Wilson also used such bounds when reducing the problem of proving that periodic profinite groups are locally finite to pro $p$-groups \cite{JW}. E. Zelmanov then solved both of these problems in his famous papers \cite{Zel1,Zel2,Zel3}.



 
If $A$ is a subgroup of $H$, let $A^H$ denote the subgroup $H$ generated by all the conjugates of $A$ in $H$.
Theorems \ref{KSConjecture1} and \ref{KSConjecture2} can in fact be deduced from the following general result.

\begin{Theorem}\label{MainTheorem} Let $G$ be a group which satisfies the max condition on subgroups, and the min condition on subnormal subgroups, and let $A$ be a subgroup of $G$ with $A^G=G$. Set $Y:=Y_{G}(A)=\langle    H \lneq G \text{ : }  A \le H\text{ and }H = A^H  \rangle$. Then one of the following holds.
\begin{enumerate}[(1)]
\item $Y=G$.
\item $A$ is contained in a unique maximal subgroup of $G$.\end{enumerate}
\end{Theorem}
Theorem \ref{MainTheorem} will follow from a generalization of a result of Flavell, which itself generalizes Wielandt's Zipper Lemma.
 

As a by-product of our methods, we also obtain strong results concerning the sets $E_{G,k}(x)$.
\begin{Theorem}\label{MainTheoremE} Let $G$ be a finite group, and let $\alpha$ be an element of $\Aut(G)$ with the property that $[G,\al]=G$. Then $G=\langle E_{G,k}(\alpha)\rangle$ for all positive integers $k$.\end{Theorem}    
In Theorem \ref{MainTheoremE}, the commutators $[g,_k\alpha]$ are understood to be computed in $\langle G,\alpha\rangle$.

An easy corollary of Theorem \ref{MainTheoremE} is the following:

\begin{Corollary}\label{KSCor} Let $G$ be a finite group, and let $x$ be an element of $G$. Let $H$ be the final term in the subnormal series $G\geq [G,x]\geq [G,x,x]\geq $, and 
let $K$ be the final term in the series $G\geq \langle E_{G,1}(x)\rangle \geq \langle E_{G,2}(x)\rangle \geq\hdots$. 
\begin{enumerate}
\item $\langle E_{G,k}(x) \rangle$ is subnormal in $G$ for all $k$; and
\item  $H = K$. 
\item The minimum of the generalized Fitting heights [respectively insoluble lengths] of the groups $\langle E_{G,k}(x)\rangle$, for $k\in\mathbb{N}$, is the generalized Fitting height [resp. insoluble length] of $K$.  
\end{enumerate} 
\end{Corollary}
We remark that Corollary \ref{KSCor} is a new result, even in the soluble case: the subnormality of the groups $\langle E_{G,k}(\al) \rangle$ was (as far as we know) previously unknown. 

Next, we note that in the special case where the automorphism $\alpha$ in Theorem \ref{MainTheoremE} is an involution, a stronger result is available. First, in this case define the set
\begin{align*} J_G(\alpha):= \{g\in G\text{ : }g\text{ has odd order and } g^{\alpha}=g^{-1}\}.
\end{align*}
We then have the following.
\begin{Theorem}\label{MainTheoremJ} Let $1\neq G$ be a finite group, and let $\alpha\in\Aut(G)$ be an involution. Suppose that $[G,\al]=G$, and write $2^k=\max\{|g|_2\text{ : } g\in G\text{ and } g^\alpha=g^{-1}\}$. Then $J_G(\al)=E_{G,j}(\al)$ for all $j>k$. In particular, $G=\langle J_G(\alpha)\rangle$ by Theorem \ref{MainTheoremE}.\end{Theorem}

We also prove the following, which may be of independent interest.
\begin{Theorem}  \label{JBound}
Let $1\neq G$ be a finite group, and let $\alpha\in\Aut(G)$ be an involution. Suppose that $[G,\al]=G$. Then 
there is some function $f$ so that $|G| \le f(|J_G(\al)|)$.
\end{Theorem}

Since $H=[H,\al]$ if and only if $H=\langle J_H(\al)\rangle$, we also deduce the following.
\begin{Corollary}  \label{cor:jbound}  Let $1\neq G$ be a finite group, and let $\alpha\in\Aut(G)$ be an involution. Then $|\langle J_G(\al) \rangle|$ is bounded above by
a function of $|J_G(\al)|$.  
\end{Corollary}  

The previous two results do not depend upon the classification of finite simple groups. 
Using the classification   \cite{GR}, we can show:

\begin{Corollary}\label{JCor} Let $1\neq G$ be a finite group, and let $\alpha\in\Aut(G)$ be an involution.  If $G=[G,\al]$, then $[G:F(G)] <   |J_G(\al)|!^4$.
\end{Corollary}  

The proof of Corollary \ref{JCor} in the case that $Z(G)=1$ follows from Theorem \ref{MainTheoremJ}.   That implies
that $C_G(\al)/Z(G)$ acts faithfully on $J^{\#} : =J_G(\al) \setminus{\{1\}}$.  Moreover $\al$ acts as a fixed point free involution
on $J^{\#}$,  whence if $m= |J^{\#}|$,  $|C_G(\al)| \le  2^{m/2-1}(m/2)!$ and then
using the main theorem in \cite{HM}, we get a bound for $[G:F(G)]$ (without the classification of finite simple groups) and the 
specific bound using \cite{GR} which depends upon the classification.   If $Z(G) \ne 1$, one
needs to modify the proof slightly (see Lemma \ref{FLemma}).

In fact, the bound in Corollary \ref{JCor}  is not so far from best possible.  Let $A = S_n$, $G = A_n$ and $\al$ a transposition in $A$.
Then $|J_G(\al)| = 2n-3$,  $|C_G(\al)|  = (n-2)!$ and $|G| > ((|J_G(\al)|-1)/2)!$.

Throughout, we write $Z(G)$, $[G,G]$, and $O(G)$ to denote the centre, derived subgroup, and largest odd order normal subgroup of $G$, respectively. For an element $g$ of $G$, we will denote the order of $g$ by $|g|$. 

Finally, we remark that the earlier bounds \cite[Theorems 1.2 and 1.3]{KS} on the generalized Fitting height and insoluble length of a finite group depend on the classification of finite simple groups, through the use of solubility of the outer automorphisms of a finite simple group (Schreier's conjecture). All results in this paper, apart from Corollary \ref{JCor} are obtained without the use of the classification of finite simple groups.

\textbf{Acknowledgement:} We would like to thank Professor Pavel Shumyatsky for introducing us to these problems, and for useful discussions. We would also like to thank the anonymous referee for their time and effort in reviewing the paper.

\section{Proofs of Theorems \ref{KSConjecture1}, \ref{KSConjecture2}, \ref{MainTheoremE} and \ref{MainTheoremJ}; and Corollary \ref{KSCor}}
We begin this section by showing that Theorems \ref{KSConjecture1}, \ref{KSConjecture2}, \ref{MainTheoremE} and \ref{MainTheoremJ}, together with Corollary \ref{KSCor}, follow from Theorem \ref{MainTheorem}.

\begin{Proposition}\label{RedProp2} Assume that Theorem \ref{MainTheorem} holds.  Then Theorems \ref{KSConjecture1}, \ref{KSConjecture2}, \ref{MainTheoremE} and \ref{MainTheoremJ}, together with Corollary \ref{KSCor}, hold.
   \end{Proposition}
   \begin{proof} We first prove that Theorem \ref{MainTheoremE} holds. So assume that $G$ is a finite group, $\alpha$ is an automorphism of $G$, and $[G,\alpha]=G$. Suppose that $k$ is minimal with the property that $\langle E_{G,k}(\alpha)\rangle\neq G$. Note that if $H$ is a subgroup of $G$, then $[H,\alpha]=H$ if and only if $\langle\alpha^H\rangle=\langle H,\alpha\rangle$. Set $X:=\langle G,\alpha\rangle$. Recall that $Y:=Y_{X}(\alpha)=\langle    H \lneq X \text{ : }  \alpha\in H = \alpha^H  \rangle$. If $Y:=Y_X(\alpha)=X$, then $X=\langle \alpha^H\text{ : }H\lneq G\rangle$ and the result follows by induction on the order of $G$. Thus, by Theorem \ref{MainTheorem} we may assume that $\alpha$ is contained in a unique maximal subgroup $K$ of $X$.   Since $\langle G, \alpha \rangle$ is generated by conjugates of $\alpha$,  the subgroup $K$ is self normalizing. 
   
Now, by the minimality of $k$ we may choose an element $h:=[g,_{k-1}\alpha]$ of $X\backslash K$. Then $[h,\alpha]\in K$, so $\alpha^h\in K$. Hence $\alpha\in K^{h^{-1}}\neq K$, contrary to assumption. This completes the proof of Theorem \ref{MainTheoremE}.

We next prove Corollary \ref{KSCor}.   Let $N$ be a minimal normal subgroup of $X$.   Let $R = \langle E_{G,k}(\alpha)\rangle$.  We need
to show that $R$ is subnormal in $X$.    By induction $RN$ is subnormal in $X$
and so it suffices to prove that $N$ normalizes $R$ or $RN=RC_N(\alpha)$ (since $C_N(\alpha)$ normalizes $R$).  

First suppose that $N$ is elementary abelian.   Then since $N$ acts trivially on $N$,  $R$ acts completely reducibly 
on $N$.   Let $M$ be an irreducible $R$-submodule of $N$ with $M \cap R = 1$  (or equivalently
$M$ is not contained in $R$).  Then 
$M_0 := [M,\alpha, \ldots, \alpha]$ (the $k$-fold commutator) is contained in $M \cap R =1$.   Thus 
$C_M(\alpha) \ne 1$ and so $C_M(\alpha)R \cap M \ne 1$.   Since $M$ is an irreducible $R$-module,
this implies that $M \le C_M(\alpha)R$.   Since this is true for any irreducible submodule,  $N \le C_N(\alpha)R$.

So we may assume that $N$ is a direct product of $t \ge 1$ non-abelian simple groups.    By  Theorem \ref{MainTheoremE}, 
$R \cap N$ contains $[N, \alpha]$, whence  $N=(N \cap R)C_N(\alpha)$   and as above, $N$ normalizes $H$.   
This proves (1) of Corollary \ref{KSCor}. Since $K$ is subnormal in $E_{G,k}(\alpha)$ for all $k\in\mathbb{N}$, part (3) also follows.  Moreover, since $H=[H,\alpha]$, Theorem \ref{MainTheoremE} implies
that $H = \langle E_{H,m}(\alpha) \rangle \le  \langle E_{G,m}(\alpha) \rangle$ for all $m$, whence $H \le K$.  
Obviously $K \le H$, whence $H=K$.

We now prove Theorem \ref{MainTheoremJ}. So assume that $\alpha\in\Aut(G)$ is an involution. Write $2^k=\max\{|g|_2\text{ : } g\in G\text{ and } g^\alpha=g^{-1}\}$, and note that $[g,_j\alpha]=[g,\alpha]^{(-2)^{j-1}}$ whenever $g\in G$. Since $[g,\al]$ is inverted by $\alpha$ for all $g\in G$, we deduce that $E_{G,j}(\al)\subseteq J_G(\al)$ whenever $j>k$. Thus, since $J_G(\al)$ is contained in  $E_{G,j}(\al)$ for all $j\geq 1$, we have $J_G(\alpha)=E_{G,j}(\alpha)$ whenever $j > k$.
 
Next, we prove that Theorems \ref{KSConjecture1} and \ref{KSConjecture2} follow from Theorem \ref{MainTheoremE}.

Indeed, suppose that $G$ is a finite group, $\alpha$ is an element of $G$, and $k$ is a positive integer. Let $H$ be the stable term in the subnormal series $G\geq [G,\al]\geq [G,\al,\al]\geq\hdots$. Then $H=[H,\al]$, so $H=\langle E_{H,k}(\al)\rangle\le \langle E_{G,k}(\al)\rangle$ by Theorem \ref{MainTheoremE}. Now, an easy exercise shows that the generalized Fitting height [respectively insoluble length] of $H^G$ coincides with the generalized Fitting height [resp. insoluble length] of $H$, since $H$ is subnormal in $G$. Moreover, the generalized Fitting height [resp. insoluble length] of $H$ is at most the generalized Fitting height [resp. insoluble length] of $\langle E_{G,k}(\al)\rangle$. Let $h$ be the generalized Fitting height [resp. insoluble length] of $H$. Then $H^G$, being a normal subgroup of $G$ with generalized Fitting height [resp. insoluble length] $h$, is contained in $F^{\ast}_h(G)$ [resp. $R_h(G)$]. Hence, $F^{\ast}_h(G)\alpha$ [resp. $R_h(G)\al$] is contained in $F(G/F^{\ast}_h(G))\le F^{\ast}_{h+1}(G)/F^{\ast}_h(G)$ [resp. $F(G/R_h(G))=1$] by Baer's Theorem.

On the other hand, suppose that $F^{\ast}_f(G)\al$ is contained in $F(G/F^{\ast}_f(G))$ [resp. $R_f(G)$], where $f\geq 0$. Then $\al$ centralizes $HF^{\ast}_f(G)/F^{\ast}_f(G)$ [resp. $HR_f(G)/R_f(G)$], so $[H,\al]=H$ implies that $H$, and hence $H^G$, is contained in $F^{\ast}_f(G)$ [resp. $R_f(G)$]. Thus, $h\le f$, as required.  
\end{proof}

We will now begin preparations for the proof of Theorem \ref{MainTheorem}. First,
for a proper subgroup $A$ of a group $H$, we define the \emph{normal closure descending series} for $A$ in $H$ as follows. Let $H_0=H$ and let $H_{i+1} = \langle A^{H_i} \rangle$. We define $F(A,H):=\cap_{i\geq 0} H_i$.

We first note the following trivial facts:
\begin{Lemma}\label{TrLemma}  Let $H$, $A$, $H_i$ and $F(A,H)$ be as defined above, and assume that $H$ satisfies the min condition on subnormal subgroups. Then
\begin{enumerate}[(i)]
\item $H_{i+1}$ is  normal in $H_i$ and $H_j$ is subnormal in $H$ for all $j$. 
\item  $ \langle A^{F(A,H)}  \rangle = F(A,H)$.
\item   If $A \le L \le H$ and $\langle A^L \rangle = L$, then $L \le F(A,H)$. 
\end{enumerate}
\end{Lemma}
\begin{proof} Parts (i) and (ii) follow immediately from the definition of the series $H=H_0\geq H_1\geq\hdots $, since $F(A,H)=H_m$ is a member of the series in this case. So assume that $A \le L \le H$ and that $\langle A^L \rangle = L$. Then $H_0=LH_0$, so $H_1=\langle A^{LH_0}\rangle\geq \langle A^{L}\rangle=L$. Extending this argument inductively yields $L\le H_i$ for all $i$. Hence, $L\le H_m=F(A,H)$.
\end{proof}

Next, we prove a generalization of a result of Flavell, which states that if $G$ is finite and $A$ is a proper subgroup of $G$ which is 
contained in at least two maximal subgroups and is subnormal in all but at most one of the maximal subgroups in which it is contained, then $A$ is contained in a proper normal subgroup of $G$. Write $\mathcal{M}(A)$ for the set of maximal subgroups of $G$ containing $A$. We remark that Wielandt's Zipper Lemma \cite[Theorem 2.9]{Is}, which is usually stated for finite groups, holds in the more general case where $G$ satisfies the max condition for subgroups, while his Join Lemma \cite[Theorem 2.5]{Is}, holds whenever $G$ satisfies the max condition for subnormal subgroups. Hence, from \cite[proof of Main Theorem]{Fl} we can see that Flavell's Theorem holds in the more general case where $G$ satisfies the max condition on subgroups.

Our generalization can now be given as follows.
\begin{Lemma}\label{FlavellGen} Let $G$ be a group satisfying the max condition on subgroups, and the min condition on subnormal subgroups. Let $A$ be a proper subgroup of $G$ satisfying the following:
\begin{enumerate}[(a)]
\item $A$ is contained in at least two maximal subgroups of $G$.
\item The set $\{F(A,H)\text{ : }H\in \mathcal{M}(A)\}$ has a unique maximal element.
\end{enumerate}
Then $A$ is contained in a proper normal subgroup of $G$.
\end{Lemma}
\begin{proof} Clearly, we may assume that $\langle A^G\rangle=G$. Denote by $\Omega(A,G)$ the set of subgroups $H$ of $G$ with the property that $\langle A^H\rangle=H$, and let $Y$ be the unique maximal element of the set $\{F(A,H)\text{ : }H\in \mathcal{M}(A)\}$. Also, let $M$ be a maximal subgroup of $G$ containing $Y$.

Now, choose $X \in \Omega(A,G)$ maximal with respect to $X$ being contained in at least two maximal subgroups of $G$.
This set is not empty since $A$ has this property. If $L\neq M$ is any maximal subgroup of $G$ containing $X$,
observe that $X$ is the stable term in the normal closure series for $A$ in $L$ (by part (iii) of the previous lemma,
$X$ is contained in the stable term, which in turn is contained in $Y\le M$ and by maximality, it is the stable term).

Thus, $X$ is subnormal in all but at most one of the maximal subgroups in which it is contained. Flavell's Theorem \cite{Fl} then implies that $X$ is contained in a proper normal subgroup of $G$. Since $A\le X$, this completes the proof.\end{proof} 

We remark that this does indeed generalize the theorem of Flavell mentioned above. To see this, suppose that $G$ is finite, $A\lneq G$, $|\mathcal{M}(A)|\geq 2$, and $A$ is subnormal in all but at most one member, say $M$, of $\mathcal{A}$. We claim that $A$ is contained in a proper normal subgroup of $G$. Clearly, we may assume that $\langle A^G\rangle=G$. We first prove that $A=F(A,L)$ for any maximal subgroup $L\neq M$ containing $A$. Let $A=H_{d+1}<H_d<\hdots<H_1<H_0=L$ be a subnormal chain for $A<L$, where $H_i\unlhd H_{i+1}$. Consider the normal closure descending series $F(A,L)=L_m<\hdots<L_1<L_0=L$ for $A$ in $L$ as defined above. Then $L_1=A^L\le H_1$, and it follows via an easy inductive argument that ${L_i}\le H_{i}$ for all $i$. Thus, $A\le F(A,L)\le L_{d+1}\le H_{d+1}=A$, so $A=F(A,L)$, as claimed. It follows that the set $\{F(A,H)\text{ : }H\in \mathcal{M}(A)\}$ has a unique maximal element. The claim follows.    
 
We are now ready to prove Theorem \ref{MainTheorem}.
\begin{proof}[Proof of Theorem \ref{MainTheorem}]  Let $\Omega(A,G)$ be as in the proof of Lemma \ref{FlavellGen}, and let $Y = \langle H | H \in \Omega(A,G) \rangle$ . Assume that $Y \neq G$, and let $M$ be a maximal subgroup of $G$ containing $Y$. Then $Y=F(A,M)$. 

Now, every subgroup in $\Omega(A,G)$ is contained in $Y$ (and so in $M$). 
Suppose that $A$ is contained in some maximal subgroup $K \neq M$. Then $F(A,K)\in\Omega(A,G)$, so $A\le F(A,K)\le M$. Hence, $F(A,K)\le Y$ by Lemma \ref{TrLemma} part (iii). It follows that the set $\{F(A,H)\text{ : }H\in \mathcal{M}(A)\}$ has a unique maximal element. Thus, $A$ is contained in a proper normal subgroup of $G$, by Lemma \ref{FlavellGen}. This contradicts $\langle A^G\rangle=G$, and completes the proof.\end{proof} 
 



\section{Proof of Theorem \ref{JBound}}\label{JSection}
In this section, we prove Theorem \ref{JBound}. So throughout, we assume that $G$ is a finite group, and that $\al$ is an involutory automorphism of $G$ with the property that $[G,\al]=G$. Set $A=\langle G, \al \rangle$ and $J=J_G(\al)$. Recall that $J$ is the set of elements of odd order in $G$ inverted
by $\al$. 

Before proceeding to the proof of Theorem \ref{JBound}, we note two useful lemmas.
\begin{Lemma}\label{lem:useful}  $\bigcap_{j\in J} C_G(\al)^j = C_G(A) \le Z(G)$.   In particular, $\core_G(C_G(\alpha))$ is contained in $Z(G)$.
\end{Lemma}

\begin{proof}   Note that if $j \in J$, then $\langle \al, \al^j \rangle$ is a dihedral group containing $j$ (since $j$ has odd order).
Thus $\langle \al^j | j \in J\rangle = \langle J, \al \rangle = A$ and the result follows. 
 \end{proof}

\begin{Lemma}\label{FLemma}
$[G:F(G)]$ can be bounded above in terms of $|J|$.
\end{Lemma}
\begin{proof}
Note first that since $G=[G,\al]$,
$\al$ acts without fixed points on $G/[G,G]$, and so acts by inversion. In particular, $|G/[G,G]|$ is odd. 

Now, there is no harm in assuming that $\Phi(G)=1$ since $G/F(G)$ does not change. It follows that $\Phi(Z(G))=1$. That is, $Z(G)$ has square-free order.
If $Z(G) \ne 1$, we can therefore write $G = M \times B$ where
$B\le Z(G)$ has prime order. Thus, $[G,G]=[M,M]\le M$ and so $|B|$ is odd and $M$ is $\al$-invariant. It also follows that $[M,\al]=M$. Furthermore, $[G:F(G)]=[M:F(M)]$. The inductive hypothesis then yields the result. 

Thus, we may assume that $Z(G)=1$.
Since $J$ generates by $G$ by Theorem \ref{MainTheoremJ}, it follows that
$C_G(\al)$ acts faithfully on $J\setminus{\{1\}}$ and so $|C_G(\al)| \le (|J|-1)!$.   The main theorem in \cite{HM} (or \cite{GR} if we allow the Classification of Finite Simple Groups) then bounds $[G:F(G)]$ in terms of $|J|$.   
\end{proof}

We are now ready to prove Theorem \ref{JBound}.
\begin{proof}[Proof of Theorem \ref{JBound}]
We prove the theorem by induction on $|J|$.    If $|J|=1$, then $G = \langle J \rangle$ is trivial and the result holds.
So we assume that $J$ and $G$ are both nontrivial. 

We first note the following elementary and well known
observation.   If $Ng \in J_{G/N}(\al)$,  with $N$ a normal subgroup of $A$ contained in $G$, then 
there is $h  \in Ng \cap J$  ($b:=[\al,g]$ is inverted by $g$ and is in the coset $Ng^2$ -- now  take $h = b^e$
for $e$ some power of $2$).  

In particular, if $N$ is a normal subgroup of $A$ contained in $G$ and $J \cap N$ is non-trivial, then 
\begin{align}\label{JGN}
 |J_{G/N}(\al)| < |J|.   
\end{align}
Now, as mentioned in the proof of Lemma \ref{FLemma}, $|G/[G,G]|$ has odd order and
$\al$ acts without fixed points and so acts by inversion.  In particular, 
\begin{align}\label{GGP}
  \text{$|G/[G,G]| \le |J|$ and $C_G(\al) \le [G,G]$. }  
\end{align}

Next, set $L:=[O(Z(G)),\al]$. If $L>1$, then since $L\subseteq J$, the inductive hypothesis and (\ref{JGN}) imply that  $|G/L|$ is bounded in terms of $|J|$. Thus, $|G|$ can be bounded in terms of $|J|$. So we may assume that 
\begin{align}\label{L}
[O(Z(G)),\al]=1.    
\end{align}

Now, by Lemma \ref{FLemma}, we may assume that $F(G) \ne 1$. If $F(G)=Z(G)$,  then again by Lemma \ref{FLemma}, 
$|G/Z(G)|$ can be bounded in terms of $|J|$.  
Schur's theorem then gives a bound for $|[G,G]|$ in terms of $|G/Z(G)|$ and again the result holds, using (\ref{GGP}).
 
Now, let $N/Z(G)$ be a minimal normal subgroup of $A/Z(G)$ contained in $F(G)/Z(G)$.
So $N/Z(G)$ is an elementary abelian $p$-group for some prime $p$.   

First suppose that $p$ is odd. 
We claim that 
\begin{align}\label{C}
\al\text{ does not centralize }N/Z(G).   
\end{align}
Indeed, since $N$ is nilpotent, we may write $N=O_p(N)Z(G)$.  
Then $[N,\al]=[O_p(N),\al] \le O_p(Z(G))$. Since
$\al$ centralizes $O(Z(G))$ by (\ref{L}), we have $[N,\al,\al]=1$. Since $O_p(N)$ is nilpotent of odd order, Baer's Theorem then implies that $\al$ centralizes $O_p(N)$. Hence, $O_p(N) \le \core_G(C_G(\al))\le Z(G)$ by Lemma \ref{lem:useful}. This contradiction proves (\ref{C}).

Now, (\ref{C}) implies that $N/Z(G) \cap J_{G/Z(G)}(\alpha)$ is non-trivial. Thus, $J\cap Z(G)n>1$ for some $n\in N$, by the assertion in the first paragraph. It follows that $p\le |J\cap N|\le |J|$ and,  from (\ref{JGN}), that $|J_{G/N}(\al)| < |J|$. Thus, $G/N$ has order bounded in terms of $|J|$. Since $N/Z(G)$ is an irreducible module for $G/N$ and $p$ is $|J|$-bounded, we deduce that $|N/Z(G)|$,  and hence $G/Z(G)$,
has order bounded order in terms of $|J|$. Applying Schur's  theorem and (\ref{GGP})   then yields a bound for $|G|$ in terms of $|J|$.

Thus, we may assume that $O_p(G) \le Z(G) $ for all odd primes $p$.   Suppose now that $p=2$. Then arguing with Baer's Theorem as above, and
using the fact that $J$ consists of elements of odd order and $G = \langle J \rangle$, we see that 
$G$ does not centralize $N/Z(G)$. Let $x \in J$ be an element of odd prime order acting non-trivially
on $N/Z(G)$.  

Let $w \in N$.  Then $(wx)^{\al} = w^{\al}x^{-1} $ and so $\al$ inverts $wx$ if and only if
$x^{-1}w^{-1} = w^{\al}x^{-1}$.  Let $V/Z(G)$ be a  $\langle \al, x\rangle$-submodule of $N/Z(G)$ with $x$
acting nontrivially.  Choose  $Z(G)v\in V/Z(G)$ so that $x$ acts non-trivially on $V$. Then set $w:=v$ if $\al x^{-1}$ centralizes $v$, and $w:=vv^{\al x^{-1}}$ otherwise. It is then a straightforward computation to see that $Z(G)wx$ has order $|x|$ and $Z(G)x^{-1}w^{-1} = Z(G)w^{\al}x^{-1}$. Thus, $Z(G)wx\in J_{G/Z(G)}(\al)\cap N/Z(G)$. 
We can then deduce from (\ref{JGN}), as in the $p$ odd case above, that $|J_{G/N}(\al)| < |J|$ and so $G/N$ is $|J|$-bounded by the inductive hypothesis.
Since $N/Z(G)$ is an irreducible module for $G/N$ and $p=2$, this implies that $|N/Z(G)|$, and hence $|G/Z(G)|$, is $|J|$-bounded. Schur's Theorem  then implies that $|G|$ is $|J|$-bounded, and this completes the proof. 
\end{proof}

\end{document}